\documentclass[1 [leqno,11pt]{amsart}
\usepackage{amssymb, amsmath}
\newcommand{\andf}{\quad\hbox{and}\quad}

\def\Supp{\mathop{\rm Supp}\nolimits\ }
\newcommand{\newcom}{\newcommand}
\def\longformule#1#2{
\displaylines{ \qquad{#1} \hfill\cr \hfill {#2} \qquad\cr } }
\def\inte#1{
\displaystyle\mathop{#1\kern0pt}^\circ }
\newcommand{\w}[1]{\langle {#1} \rangle}
\newcom{\al}{\alpha}
\newcom{\de}{\delta}
\newcom{\be}{\beta}
\newcom{\s}{\sigma}
\newcom{\eps}{\epsilon}
\newcom{\ve}{\varepsilon}
\newcom{\ga}{\gamma}
\newcom{\Ga}{\Gamma}
\newcom{\ka}{\kappa}
\newcom{\Lam}{\Lambda}
\newcom{\lam}{\lambda}
\newcom{\vp}{\varphi}

\newcom{\Om}{\Omega}
\newcom{\om}{\omega}
\newcom{\Sig}{\Sigma}
\newcom{\sig}{\sigma}
\newcom{\tht}{\theta}
\newcom{\tri}{\triangle}
\newcom{\oo}{\infty}
\newcom{\h}{{\rm h}}
\newcom{\vphi}{\varphi}
\newcom{\cB}{{\mathcal B}}
\newcom{\cC}{{\mathcal C}}
\newcom{\cD}{{\mathcal D}}
\newcom{\cF}{{\mathcal F}}
\newcom{\cL}{{\mathcal L}}
\newcom{\cM}{{\mathcal M}}
\newcom{\cP}{{\mathcal P}}
\newcom{\cS}{{\mathcal S}}
\newcom{\cQ}{{\mathcal Q}}
\newcom{\cT}{{\mathcal T}}
\newcom{\cY}{{\mathcal Y}}
\newcom{\cZ}{{\mathcal Z}}
\newcom{\R}{\Bbb R}
\newcom{\T}{\Bbb T}
\newcom{\N}{\Bbb N}
\newcom{\Z}{\Bbb Z}
\newcom{\C}{\Bbb C}
\newcom{\E}{\Bbb E}
\let\wh=\widehat
\def\dive{\mathop{\rm div}\nolimits}
\let\e=\varepsilon

\newcom{\f}{\frac}
\newcom{\dint}{\displaystyle\int}
\newcom{\dsum}{\displaystyle\sum}
\newcom{\dlim}{\displaystyle\lim}
\newcom{\ov}{\overline}
\newcom{\wt}{\widetilde}
\newcom{\pa}{\partial}
\newcom{\p}{\partial}
\newcom\na{\nabla}
\newcom{\D}{\Delta}
\newcom\rto{\rightarrow}
\newcom\lto{\leftarrow}
\newcom\mto{\mapsto}
\newcom{\disp}{\displaystyle}
\newcom{\non}{\nonumber}
\newcom{\no}{\noindent}
\newcom{\QED}{$\square$}

\def\eqdefa{\buildrel\hbox{\footnotesize def}\over =}

\newcommand{\beq}{\begin{equation}}
\newcommand{\eeq}{\end{equation}}
\newcommand{\ben}{\begin{eqnarray}}
\newcommand{\een}{\end{eqnarray}}
\newcommand{\beno}{\begin{eqnarray*}}
\newcommand{\eeno}{\end{eqnarray*}}

\newtheorem{Def}{Definition}[section]

\newtheorem{lem}{Lemma}[section]
\newtheorem{rmk}{Remark}[section]

\renewcommand{\theequation}{\thesection.\arabic{equation}}

\newtheorem{theorem}{Theorem}[section]

\newtheorem{lemma}[theorem]{Lemma}

\begin{document}
\title[long time well-posedness of Prandtl system]
{Long time well-posdness of Prandtl system with small and analytic
initial data}

\author[P. ZHANG]{Ping Zhang}%
\address[P. ZHANG]
 {Academy of
Mathematics $\&$ Systems Science and  Hua Loo-Keng Key Laboratory of
Mathematics, The Chinese Academy of Sciences\\
Beijing 100190, CHINA } \email{zp@amss.ac.cn}
\author[Z. ZHANG]{Zhifei Zhang}\address[Z. ZHANG]
{School of  Mathematical Science, Peking University, Beijing 100871,
P. R. CHINA} \email{zfzhang@math.pku.edu.cn}

\date{4. Sept. 2014}

\begin{abstract} In this paper, we investigate the long time
existence and uniqueness of small solution to $d,$ for $d=2,3,$
dimensional Prandtl system with small initial data  which is
analytic in the horizontal variables. In particular, we prove that
 $d$ dimensional Prandtl system has a unique solution with the  life-span of which  is greater than
$\e^{-\f43}$ if both the initial data and the value on the boundary
of the tangential velocity of the outflow are of size $\e.$  We
mention that the tool developed in \cite{Ch04, CGP}  to make the
analytical type estimates and the special structure of the nonlinear
terms to this system play an essential role in the proof of this
result.\vspace{0.2cm}

\vspace{0.1cm}

\end{abstract}

\maketitle

\renewcommand{\theequation}{\thesection.\arabic{equation}}
\setcounter{equation}{0}
\section{Introduction}\label{sect1}

In this paper, we investigate long time well-posedness to the
following  Prandtl system in $\R_+\times\R^d_+ ,$ for $d=2,3,$ with
small and analytic initial data:
\begin{equation}\label{1.1}
\left\{
\begin{array}{ll}
\p_tu+u\cdot\na_{\rm h}u+v\p_yu-\p_{yy}u+\na_hp=0, \quad (t,x,y)\in\R_+\times\R^{d-1}\times\R_+, \\
\dive_{\rm h} u+\p_yv =0,\\
u|_{y=0}=0,\ \ v|_{y=0}=0,\quad \mbox{and}\quad \displaystyle\lim_{y\to\infty}u(t,x,y)=U(t,x),\\
u|_{t=0}=u_{0},
\end{array}
\right.
\end{equation}
where $u=u, \na_\h=\dive_{\rm h}=\p_x$ for $d=2,$ and
$u=(u^1,u^{2}), \na_{\rm h}=(\pa_{x_1}, \pa_{x_2}),$ $\dive_{\rm
h}=\pa_{x_1}+ \pa_{x_2}$ for $d=3.$ $(u,v)$ denotes the tangential
and normal velocities of the boundary layer flow, $\big(U(t,x),
p(t,x)\big)$ are the values on the boundary of the tangential
velocity and pressure of the outflow, which satisfies Bernoulli's
law \beno \p_tU+U\cdot\na_hU+\na_{\rm h}p=0\quad \mbox{in}\quad
\R_+\times\R^{d-1}. \eeno This system proposed by Prandtl \cite{Pra}
is a  model equation for the first  order approximation of the
velocity field near the boundary in the zero viscosity limit of the
initial boundary problem of  Navier-Stokes equations with the
non-slip boundary condition. One may check \cite{Olei, E} and
references therein for more introductions on boundary layer theory.
\smallskip

One of the key step to rigorously justify this inviscid limit of
Navier-Stokes system with non-slip boundary condition is to deal
with the well-posedness of the Prandtl system. Since there is no
horizontal diffusion in the $u$ equation of \eqref{1.1}, the
nonlinear term $v\p_y u$ (which almost behaves like $\p_xu\p_y u$)
loses one horizontal derivative in the process of energy estimate,
and therefore the question of whether or not the Prandtl system with
general data is well-posed in Sobolev spaces is still open.
Recently, G\'{e}rard-Varet and Dormy \cite{Ger1} proved the
ill-posedness in Sobloev spaces for the linearized Prandtl system
around non-monotonic shear flows. The nonlinear ill-posedness was
also established in \cite{Ger2, Guo} in the sense of non-Lipschtiz
continuity of the flow. Nevertheless, we have the following positive
results for two classes of special data.

$\bullet$ Under a monotonic assumption on the tangential velocity of
the outflow and $d=2$, Oleinik \cite{Olei} proved the local
existence and uniqueness of classical solutions to \eqref{1.1}. With
the additional ``favorable" condition on the pressure, Xin and Zhang
\cite{Xin} obtained the global existence of weak solutions to this
system. The main idea in \cite{Olei, Xin} is to use  Crocco
transformation. We refer to \cite{Alex, Mas} for recent  proofs and
generalizations of such kind of results which is  based on the
direct energy method.

$\bullet$ For the data which is analytic in $x,y$ variables,
Sammartino and Caflisch \cite{Caf} established the local
well-posedness result of \eqref{1.1}. Later, the analyticity in $y$
variable was removed by Lombardo, Cannone and Sammartino in
\cite{Can}. The main argument used in  \cite{Caf, Can} is to apply
the abstract Cauchy-Kowalewskaya (CK) theorem. We also mention a
recent well-posedness result of \eqref{1.1} for a class of data with
Gevrey regularity \cite{GM}. \smallskip

On the other hand, Chemin, Gallagher and Paicu \cite{CGP} (see also
\cite{mz1,mz2}) proved an interesting result concerning the global
well-posedness of three dimensional Navier-Stokes system with a
class of ``ill prepared data", which is slowly varying in the
vertical variable, namely of the form $\e x_3,$  and the critical
norm of which blows up as the small parameter goes to zero. The main
idea of the proof in \cite{CGP} ( see also \cite{mz1, mz2}) is that:
after a change of scale, one obtains anisotropic Navier-Stokes
system which has such diffusion term as $\D_\h+\e^3\p_{x_3}^2$ and
anisotropic pressure gradient of the form $-(\na_\h p,
\e^2\p_{x_3}p),$  therefore there is a loss of regularity in the
vertical variable in the classical Sobolev estimates and it is
natural to work this transformed problem with initial data in the
analytical spaces. However the main disadvantage of CK type argument
is that it does not provide either the explicit radius of
analyticity or the lifespan of the solution.  The main idea to
overcome those difficulties is to use the tool introduced by Chemin
\cite{Ch04} which consists in making analytical type estimates and
controlling the size of the analytic band simultaneously. One may
check \cite{Ch04,CGP, mz1, mz2} and the references therein
concerning related results on analytic solutions to classical
Navier-Stokes system. \smallskip

Motivated by \cite{Ch04, CGP, mz1, mz2}, we are going to investigate
the long time   well-posedness of Prandtl system with small and
analytic initial data. Since there is no horizontal diffusion in the
System \eqref{1.1}, this system looks like a  hyperbolic one. Hence,
it is natural to expect that the lifespan of the solutions to this
problem with initial data of size $\e$ should be of size
$O(\ve^{-1})$. Indeed, this is possible if we go through the proofs
of classical CK type argument step by step.  By making full use of
the vertical diffusion in \eqref{1.1} as well as $\dive_\h
u+\p_yv=0,$ we shall prove that vertical diffusion   slows down the
decreasing  of the analyticity radius so that the lifespan of the
solution becomes longer.

We remark that since the $y$ variable only lies in the upper half
line, it is natural to use the $L^2$ framework in process of
energy estimate of \eqref{1.1}. However, note that this approach
only gains half derivative, and to solve the well-posedness of
\eqref{1.1} (see \cite{Ch04}), one needs one more derivative in the
horizontal variables. In order to gain this additional one
derivative, we will have to use the weighted Chemin-Lerner spaces
introduced by Paicu and Zhang in \cite{PZ1}.

For simplicity, here we just consider the case of uniform outflow
where $U=\ve \bf{e}$ for some unit vector ${\bf e}\in \R^{d-1}$,
which implies $\na_{\rm h}p=0.$ Let $u^s(t,y)$ be determined by
\begin{equation}\label{1.2}
 \quad\left\{\begin{array}{l}
\displaystyle \p_tu^s-\p_{yy}u^s=0, \quad (t,y)\in\R_+\times\R_+,\\
\displaystyle u^s|_{y=0}=0\andf \lim_{y\to\infty}u^s(t,y)=\ve{\bf e}, \\
\displaystyle u^s|_{t=0}=\ve\chi(y)\bf{e},
\end{array}\right.
\end{equation}
where $\chi(y)\in C^\infty(\R)$, and $\chi(y)=0$ for $y\le 1$ and $\chi(y)=1$ for $y\ge 2$.

By substituting $u=u^s+w$ in \eqref{1.1} and using \eqref{1.2}, we
write \begin{equation}\label{1.3} \left\{
\begin{array}{ll}
\displaystyle \p_t
w+(w+u^s)\cdot\na_\h w-\int_0^y\dive_\h w\,dy'\p_yw-\int_0^y\dive_\h w\,dy'\p_yu^s-\p_{yy}w=0,\\
\displaystyle w|_{y=0}=0\quad\mbox{and}\quad \lim_{y\to\infty} w=0, \\
\displaystyle w|_{t=0}=u_0-\ve\chi{\bf e} \eqdefa w_0.
\end{array}\right.
\end{equation}

 Our main result is
stated as follows.

\begin{theorem}\label{th1.1}
{\sl Let $\delta>0$ and $\ve$ be a sufficiently small positive
constant. Assume that $w_0$ satisfies \beq\label{eq1.4}
\|e^{\f{1+y^2}8}e^{\delta|D|}w_0\|_{\cB^{\f{d-1}2,0}}\leq \ve, \eeq
then there exits a positive time $T_\ve$ which is of size greater
than $ \ve^{-\f43}$ so that (\ref{1.3}) has a  unique solution $w$
which satisfies \ben\label{f-class} e^{\Psi(t,y)}e^{\Phi(t,D)}w \in
\widetilde{L}^\infty_T(\cB^{\f{d-1}2,0}),\quad e^{\Psi(t,y)}
e^{\Phi(t,D)}\pa_y w\in \widetilde{L}^2_T(\cB^{\f{d-1}2,0}), \een
for any $T\leq T_\ve,$ and where  the functions $ \Psi(t,y),
\Phi(t,\xi)$ are determined by \eqref{eq2.6} and \eqref{eq2.7}
respectively.}
\end{theorem}

The definitions of the function spaces will be presented in Section
\ref{sect2}.

\begin{rmk}\label{rmk1.1} We make the following comments concerning this
theorem:

\begin{itemize}

\item[(1)] We remark that here we require our initial data to be analytic only
in the $x$ variables. Moreover, our method to prove Theorem
\ref{th1.1} also ensures the local well-posedness of \eqref{1.3}
with general analytic initial data of arbitrary size.

\item[(2)]
In view of the result by  E and Enquist \cite{EE} concerning  the
finite time blow-up  of classical solution to \eqref{1.1}, one may
not expect the global existence result for \eqref{1.1} with general
data without monotonicity assumption in \cite{Olei}. However,
whether the lifespan obtained in Theorem \ref{th1.1} is sharp  is a
very interesting question.

\item[(3)]
The condition (\ref{eq1.4}) can be relaxed to \beno
\|e^{\rho(1+y)}e^{\delta|D|}w_0\|_{\cB^{\f{d-1}2,0}}\leq \ve, \eeno
if we take $\Psi(t,y)=\rho\Big(\f {1+y} {(1+t)^\f12}\Big)$, where
$\rho(z)\in W^{2,\infty}(\R_+)$ is a linear function of $z$ for
$z\ge 1$ and satisfies \beno \rho'(z)z\ge
2\big(\rho''(z)+2\rho'(z)^2\big). \eeno

\end{itemize}

\end{rmk}

This paper is organized as follows. In the second section, we recall
some basic facts on Littlewood-Paley theory and function spaces we
are going to use. In the third section, we present the proof to the
existence part of Theorem \ref{th1.1}. In the fourth section, we
complete the uniqueness part of Theorem \ref{th1.1}.

\medbreak Let us end this introduction by the notations we shall use
in this context.\vspace{0.2cm}

For~$a\lesssim b$, we mean that there is a uniform constant $C,$
which may be different on different lines but be independent of
$\ve,$ such that $a\leq Cb$. $(a\ |\
b)_{L^2_+}\eqdefa\int_{\R^d_+}a(x,y)\bar{b}(x,y)\,dx\,dy$ stands for
the $L^2$ inner product of $a,b$ on $\R^d_+.$ For $X$ a Banach space
and $I$ an interval of $\R,$ we denote by $L^q(I;\,X)$ the set of
measurable functions on $I$ with values in $X,$ such that
$t\longmapsto\|f(t)\|_{X}$ belongs to $L^q(I)$ and
$L^q_+=L^q(\R^{d-1}\times\R_+).$ In particular,  we denote by
$L^p_T(L^q_{\rm h}(L^r_{\rm v}))$ the space $L^p([0,T];
L^q(\R_{x_{}}^{d-1};L^r(\R_{y}^+))).$  Finally, we denote
 $
\bigl\{d_{k}\bigr\}_{k\in\Z}$  to be  generic elements in the sphere
of $\ell^1(\Z)$.

\renewcommand{\theequation}{\thesection.\arabic{equation}}
\setcounter{equation}{0}

\section{Littlewood-Paley theory and functional framework}\label{sect2}

In the rest of this paper, we shall frequently  use Littlewood-Paley
decomposition in the horizontal variables $x$. Let us recall from
\cite{BCD} that \beq
\begin{split}
&\Delta_k^{\rm h}a=\cF^{-1}(\varphi(2^{-k}|\xi|)\widehat{a}),\qquad
S^{\rm h}_ka=\cF^{-1}(\chi(2^{-k}|\xi|)\widehat{a}),
\end{split} \label{1.3a}\eeq where and in all that follows, $\cF
a$ and $\widehat{a}$ always denote the partial  Fourier transform of
the distribution $a$ with respect to $x$ variables,  that is, $
\widehat{a}(\xi,y)=\cF_{x\to\xi}(a)(\xi,y),$
  and $\chi(\tau),$ ~$\varphi(\tau)$ are
smooth functions such that
 \beno
&&\Supp \varphi \subset \Bigl\{\tau \in \R\,/\  \ \frac34 \leq
|\tau| \leq \frac83 \Bigr\}\andf \  \ \forall
 \tau>0\,,\ \sum_{j\in\Z}\varphi(2^{-j}\tau)=1,\\
&&\Supp \chi \subset \Bigl\{\tau \in \R\,/\  \ \ |\tau|  \leq
\frac43 \Bigr\}\quad \ \ \ \andf \  \ \, \chi(\tau)+ \sum_{j\geq
0}\varphi(2^{-j}\tau)=1.
 \eeno

Let us also  recall the functional spaces we are going to use.

\begin{Def}\label{def1.2}
{\sl  Let~$s$ in~$\R$. For~$u$ in~${\mathcal S}_h'(\R^d_+),$ which
means that $u$ is in~$\cS'(\R^d_+)$ and
satisfies~$\lim_{k\to-\infty}\|S_k^{\rm h}u\|_{L^\infty}=0,$ we set
$$
\|u\|_{\cB^{s,0}}\eqdefa\big\|\big(2^{ks}\|\Delta_k^{\rm h}
u\|_{L^{2}_+}\big)_k\bigr\|_{\ell ^{1}(\Z)}.
$$
\begin{itemize}

\item
For $s\leq \frac{d-1}{2}$, we define $ \cB^{s,0}(\R^d_+)\eqdefa
\big\{u\in{\mathcal S}_h'(\R^d_+)\;\big|\; \|
u\|_{\cB^{s,0}}<\infty\big\}.$

\item
If $k$ is  a positive integer and if~$\frac{d-1}{2}+k< s\leq
\frac{d+1}{2}+k$, then we define~$ \cB^{s,0}(\R^d_+)$  as the subset
of distributions $u$ in~${\mathcal S}_h'(\R^d_+)$ such that
$\na_\h^\beta u$ belongs to~$ \cB^{s-k,0}(\R^d_+)$ whenever
$|\beta|=k.$
\end{itemize}
}
\end{Def}

In  order to obtain a better description of the regularizing effect
of the transport-diffusion equation, we need to use Chemin-Lerner
type spaces $\widetilde{L}^{\lambda}_T(\cB^{s,0}(\R^d_+))$.
\begin{Def}\label{def2.2}
{\sl Let $p\in[1,\,+\infty]$ and $T\in]0,\,+\infty]$. We define
$\widetilde{L}^{p}_T(\cB^{s,0}(\R^d_+))$ as the completion of
$C([0,T]; \,\cS(\R^d_+))$ by the norm
$$
\|a\|_{\widetilde{L}^{p}_T(\cB^{s,0})} \eqdefa \sum_{k\in\Z}2^{ks}
\Big(\int_0^T\|\Delta_k^{\rm h}\,a(t) \|_{L^2_+}^{p}\,
dt\Big)^{\frac{1}{p}}
$$
with the usual change if $r=\infty.$ }
\end{Def}

In order to overcome the difficulty that one can not use Gronwall's
type argument in the framework of Chemin-Lerner space
$\wt{L}^2_t(\cB^{s,0}),$ we also need to use the weighted
Chemin-Lerner norm, which was introduced by Paicu and Zhang in
\cite{PZ1}.

\begin{Def}\label{def1.1} {\sl Let $f(t)\in L^1_{\mbox{loc}}(\R_+)$
be a nonnegative function. We define \beq \label{1.4}
\|a\|_{\wt{L}^p_{t,f}(\cB^{s,0})}\eqdefa
\sum_{k\in\Z}2^{ks}\Bigl(\int_0^t f(t')\|\D_k^{\rm
h}a(t')\|_{L^2_+}^p\,dt'\Bigr)^{\f1p}. \eeq}
\end{Def}

 \medbreak
For the convenience of the readers, we recall the following anisotropic
Bernstein type lemma from \cite{CZ1, Pa02}:

\begin{lem} \label{lem:Bern}
 {\sl Let $\cB_{\rm h}$ be a ball
of~$\R^{d-1}_{\rm h}$, and~$\cC_{\rm h}$  a ring of~$\R^{d-1}_{\rm
h}$; let~$1\leq p_2\leq p_1\leq \infty$ and ~$1\leq q\leq \infty.$
Then there holds:

\smallbreak\noindent If the support of~$\wh a$ is included
in~$2^k\cB_{\rm h}$, then
\[
\|\partial_{x}^\alpha a\|_{L^{p_1}_{\rm h}(L^{q}_{\rm v})} \lesssim
2^{k\left(|\al|+(d-1)\left(\f 1 {p_2}-\f 1 {p_1}\right)\right)}
\|a\|_{L^{p_2}_{\rm h}(L^{q}_{\rm v})}.
\]

\smallbreak\noindent If the support of~$\wh a$ is included
in~$2^k\cC_{\rm h}$, then
\[
\|a\|_{L^{p_1}_{\rm h}(L^{q}_{\rm v})} \lesssim
2^{-kN}\sup_{|\al|=N} \|\partial_{x}^\al a\|_{L^{p_1}_{\rm
h}(L^{q}_{\rm v})}.
\]
}
\end{lem}

We shall constantly use the Bony's decomposition (see \cite{Bo}) for
the horizontal variables: \ben\label{Bony} fg=T^{\rm h}_fg+T^{\rm
h}_{g}f+R^{\rm h}(f,g), \een where \beno T^{\rm h}_fg=\sum_kS^{\rm
h}_{k-1}f\Delta_k^{\rm h}g,\quad R^{\rm
h}(f,g)=\sum_k\widetilde{\Delta}_k^{\rm h}f\Delta_{k}^{\rm h}g \eeno
with $\widetilde{\Delta}_k^{\rm h}f\eqdefa
\displaystyle\sum_{|k-k'|\le 1}\Delta_{k'}^{\rm h}f$.

As in \cite{Ch04,CGP, mz1, mz2}, for any locally bounded function
$\Phi$ on $\R^+\times \R^{d-1}$, we define \beq\label{eq2.4}
u_\Phi(t,x,y)=\cF_{\xi\to
x}^{-1}\bigl(e^{\Phi(t,\xi)}\widehat{u}(t,\xi,y)\bigr). \eeq We
introduce a key quantity $\theta(t)$ to describe the evolution of
the analytic band of $w:$
\begin{equation}\label{1.9}
 \quad\left\{\begin{array}{l}
\displaystyle \dot{\tht}(t)=\w{t}^\f14\big(\|e^\Psi \p_yw_\Phi(t)\|_{\cB^{\f{d-1}2,0}}+\|e^{\Psi}\p_yu^s(t)\|_{L^2_{\rm v}}\big),\\
\displaystyle \tht|_{t=0}=0.
\end{array}\right.
\end{equation}
Here $\w{t}=1+t,$  the phase function $\Phi$ is defined by
\beq\label{eq2.6} \Phi(t,\xi)\eqdefa (\de-\lam \tht(t))|\xi|, \eeq
and the weighted function $\Psi(t,y)$ is determined by \beq
\label{eq2.7} \Psi(t,y)\eqdefa \f{1+y^2}{8\w{t}}, \eeq which
satisfies \beq\label{eq2.8} \p_t\Psi(t,y)+2(\p_y\Psi(t,y))^2\leq 0.
\eeq

\setcounter{equation}{0}
\section{The proof of the existence part of Theorem \ref{th1.1}}

The general strategy to prove the existence result for a nonlinear
partial differential equation is first to construct an appropriate
approximate solutions, then perform uniform estimates for such
approximate solution sequence, and finally pass to the limit of the
approximate problem. For simplicity, here we only present the {\it a
priori} estimates for smooth enough solutions of \eqref{1.3} in the
analytical framework.

In view of \eqref{1.3}, \eqref{eq2.4} and \eqref{eq2.6},  it is easy
to observe that $w_\Phi$ verifies \beq\label{1.6}
\begin{split}
&\p_tw_\Phi+\lam\dot{\tht}(t)|D_h|w_\Phi+[w\cdot\na_\h w]_\Phi-\bigl[\int_0^y\dive_\h w\,dy'\p_yw\bigr]_\Phi\\
&\qquad\qquad\qquad\quad+u^s\cdot\na_\h w_\Phi-\int_0^y\dive_\h
w_\Phi\,dy'\p_yu^s-\p_{yy}w_\Phi=0, \end{split} \eeq where $|D_\h|$
denotes the Fourier multiplier with symbol $|\xi|.$

Let $\Phi(t,\xi)$ and $\Psi(t,y)$ be given by \eqref{eq2.6} and
\eqref{eq2.7} respectively. By applying the dyadic operator
$\D_k^{\rm h}$ to \eqref{1.6} and then taking the $L^2_+$ inner
product of the resulting equation with $e^{2\Psi}\D_k^{\rm
h}w_\Phi,$  we find \beq \label{1.8}
\begin{split}
\bigl(&e^\Psi\p_t\D_k^{\rm h}w_\Phi\ |\ e^\Psi\D_k^{\rm
h}w_\Phi\bigr)_{L^2_+}+\lam\dot{\tht}\bigl(e^\Psi|D_h|\D_k^{\rm
h}w_\Phi\ |\ e^\Psi\D_k^{\rm
h}w_\Phi\bigr)_{L^2_+}\\
&-\bigl(e^\Psi\p_{yy}\D_k^{\rm h}w_\Phi\ |\ e^\Psi\D_k^{\rm
h}w_\Phi\bigr)_{L^2_+}+\bigl(e^\Psi\D_k^{\rm h}[w\cdot\na_\h
w]_\Phi\ |\ e^\Psi\D_k^{\rm
h}w_\Phi\bigr)_{L^2_+}\\
&-\bigl(e^\Psi\D_k^{\rm h}\bigl[\int_0^y\dive_\h
w\,dy'\p_yw\bigr]_\Phi\ |\ e^\Psi\D_k^{\rm
h}w_\Phi\bigr)_{L^2_+}+\bigl(e^\Psi u^s\cdot\na_h\D_k^{\rm h}w_\Phi\
|\ e^\Psi\D_k^{\rm
h}w_\Phi\bigr)_{L^2_+}\\
&\qquad\qquad\qquad\qquad\qquad\qquad\qquad-\bigl(e^\Psi\int_0^y\dive_\h\D_k^{\rm
h}w_\Phi\,dy'\p_yu^s\ |\ e^\Psi\D_k^{\rm h}w_\Phi\bigr)_{L^2_+}=0.
\end{split}
\eeq In what follows, we shall always assume that $t<T^\ast$ with
$T^\ast$ being determined by \beq\label{1.8a} T^\ast\eqdefa
\sup\bigl\{\ t>0,\ \ \tht(t) <\de/\lam\bigr\}. \eeq So that by
virtue of \eqref{eq2.6}, for any $t<T^\ast,$ there holds the
following convex inequality \beq\label{1.8bb} \Phi(t,\xi)\leq
\Phi(t,\xi-\eta)+\Phi(t,\eta)\quad\mbox{for}\quad \forall\
\xi,\eta\in \R^{d-1}. \eeq

Let us now handle term by term in \eqref{1.8}.\vspace{0.2cm}

\no $\bullet$ \underline{Estimate of
$\int_0^t\bigl(e^\Psi\p_t\D_k^{\rm h}w_\Phi\ |\ e^\Psi\D_k^{\rm
h}w_\Phi\bigr)_{L^2_+}\,dt'$}\vspace{0.2cm}

We first get, by using integration by parts, that
\beno
\begin{split}
\bigl(e^\Psi\p_t\D_k^{\rm h}w_\Phi\ |\ e^\Psi\D_k^{\rm
h}w_\Phi\bigr)_{L^2_+}=\bigl(\p_t(e^\Psi\D_k^{\rm h}w_\Phi)\ |\
e^\Psi\D_k^{\rm h}w_\Phi\bigr)_{L^2_+}-\bigl(\pa_t\Psi(e^\Psi\D_k^{\rm
h}w_\Phi)\ |\ e^\Psi\D_k^{\rm h}w_\Phi\bigr)_{L^2_+},
\end{split}
\eeno Integrating the above equality over $[0,t]$ gives rise to
\beq\label{1.11}
\begin{split}
\int_0^t\bigl(e^\Psi\p_t\D_k^{\rm h}w_\Phi\ |\ e^\Psi\D_k^{\rm
h}w_\Phi\bigr)_{L^2_+}\,dt'=&\f12\|e^\Psi\D_k^{\rm
h}w_\Phi(t)\|_{L^2_+}^2-\f12\|e^{\f{1+y^2}8}\D_k^{\rm
h}e^{\delta|D|}w_0\|_{L^2_+}^2\\
&-\int_0^t\int_{\R^d_+}\pa_t\Psi|e^\Psi\D_k^{\rm
h}w_\Phi(t')|^2\,dx\,dy\,dt'.
\end{split}
\eeq

\no $\bullet$ \underline{Estimate of $\int_0^t\bigl(e^\Psi\D_k^{\rm
h}[w\cdot\na_hw]_\Phi\ |\ e^\Psi\D_k^{\rm h}w_\Phi\bigr)_{L^2_+}\,dt'$}\vspace{0.2cm}

Applying Bony's decomposition \eqref{Bony} for $w\cdot\na_\h w$ for
the $x$ variables gives \beno w\cdot\na_\h w=T^{\rm h}_w\na_\h
w+T^{\rm h}_{\na_\h w}w+R^{\rm h}(w,\na_\h w). \eeno Considering
\eqref{1.8bb} and the support properties to the Fourier transform of
the terms in $T^{\rm h}_w\na_\h w,$ we write
$$\longformule{ \int_0^t\bigl|\bigl(e^\Psi\D_k^{\rm h}\bigl[T^{\rm
h}_w\na_\h w\bigr]_\Phi\ |\ e^\Psi\D_k^{\rm
h}w_\Phi\bigr)_{L^2_+}\bigr|\,dt'}{{}\lesssim \sum_{|k'-k|\leq
4}\int_0^t\|S_{k'-1}^{\rm
h}w_\Phi(t')\|_{L^\infty_+}\|e^{\Psi}\D_{k'}^{\rm h}\na_\h
w_\Phi(t')\|_{L^2_+}\|e^{\Psi}\D_{k}^{\rm
h}w_\Phi(t')\|_{L^2_+}\,dt'.} $$ However note that $w|_{y=0}=0,$ one
has \beno \|S_{k'-1}^{\rm
h}w_\Phi(t')\|_{L^\infty_+}=&\bigl\|S_{k'-1}^{\rm
h}(\int_0^y\p_yw_\Phi(t')\,dy')\|_{L^\infty_+} \leq
\|\p_yw_\Phi(t')\|_{L^1_{\rm v}(L^\infty_{\rm h})}, \eeno whereas
applying Lemma \ref{lem:Bern}  and using \eqref{eq2.7} yields \beno
\begin{split}
\|\p_yw_\Phi(t')\|_{L^1_{\rm v}(L^\infty_{\rm h})}\lesssim &
\sum_{k\in\Z} 2^{\bigl(\f{d-1}2\bigr)k}\|\D_k^\h \p_y
w_\Phi(t')\|_{L^1_{\rm
v}(L^2_\h)}\\
\lesssim & \|e^{-\Psi(t')}\|_{L^2_{\rm
v}}\sum_{k\in\Z}2^{\bigl(\f{d-1}2\bigr)k}\|e^{\Psi}\D_k^\h \p_y
w_\Phi(t')\|_{L^2_+}\\
\lesssim &\w{t'}^{\f14}\|e^\Psi\p_yw_\Phi(t')\|_{\cB^{\f{d-1}2,0}},
\end{split}
\eeno which together with \eqref{1.9} ensures that \beq\label{1.12}
\begin{split} \|S_{k'-1}^{\rm h}w_\Phi(t')\|_{L^\infty_+}
\lesssim
 \dot{\tht}(t'). \end{split} \eeq
Whence  in view of Definition \ref{def1.1} and by applying
H\"older's inequality, we obtain \beq\label{1.12a}
\begin{split}
\int_0^t\bigl|&\bigl(e^\Psi\D_k^{\rm h}\bigl[T^{\rm h}_w\na_\h
w\bigr]_\Phi\ |\ e^\Psi\D_k^{\rm
h}w_\Phi\bigr)_{L^2_+}\bigr|\,dt'\\
&\lesssim \sum_{|k'-k|\leq
4}2^{k'}\Bigl(\int_0^t\dot{\tht}(t')\|e^\Psi\D_{k'}^{\rm
h}w_\Phi(t')\|_{L^2_+}^2\,dt'\Bigr)^{\f12}\Bigl(\int_0^t\dot{\tht}(t')\|e^\Psi\D_k^{\rm
h}w_\Phi(t')\|_{L^2_+}^2\,dt'\Bigr)^{\f12}\\
&\lesssim d_k^22^{-(d-1)k}\|e^\Psi
w_\Phi\|_{\wt{L}^2_{t,\dot{\tht}(t)}(\cB^{\f{d}2,0})}^2.
\end{split}
\eeq

Along the same line, it follows from  Lemma \ref{lem:Bern} and
\eqref{1.12} that \beno \|S_{k'-1}^{\rm h}\na_\h
w_\Phi(t)\|_{L^\infty_+}\lesssim 2^{k'}\|S_{k'-1}^{\rm
h}w_\Phi(t)\|_{L^\infty_+}\lesssim 2^{k'}\dot{\tht}(t), \eeno from
which and
$$\longformule{
\int_0^t\bigl|\bigl(e^\Psi\D_k^{\rm h}\bigl[T^{\rm h}_{\na_\h
w}w\bigr]_\Phi\ |\ e^\Psi\D_k^{\rm
h}w_\Phi\bigr)_{L^2_+}\bigr|\,dt'}{{}\lesssim \sum_{|k'-k|\leq
4}\int_0^t\|S_{k'-1}^{\rm h}\na_\h
w_\Phi(t')\|_{L^\infty_+}\|e^{\Psi}\D_{k'}^{\rm
h}w_\Phi(t')\|_{L^2_+}\|e^{\Psi}\D_{k}^{\rm
h}w_\Phi(t')\|_{L^2_+}\,dt',} $$ we  thus deduce by a similar
derivation of \eqref{1.12a} that \beno
\int_0^t\bigl|\bigl(e^\Psi\D_k^{\rm h}\bigl[T^{\rm h}_{\na_\h
w}w\bigr]_\Phi\ |\ e^\Psi\D_k^{\rm
h}w_\Phi\bigr)_{L^2_+}\bigr|\,dt'\lesssim d_k^22^{-(d-1)k}\|e^\Psi
w_\Phi\|_{\wt{L}^2_{t,\dot{\tht}(t)}(\cB^{\f{d}2,0})}^2. \eeno

Finally due to \eqref{1.8bb} and the support properties to the
Fourier transform of the terms in $R^{\rm h}(w,\na_\h w),$ we write
\beno
\begin{split}
\int_0^t\bigl|&\bigl(e^\Psi\D_k^{\rm h}\bigl[R^{\rm h}(w,\na_\h
w)\bigr]_\Phi\ |\ e^\Psi\D_k^{\rm
h}w_\Phi\bigr)_{L^2_+}\bigr|\,dt'\\
&\lesssim 2^{{\bigl(\f{d-1}2\bigr)k}}\sum_{k'\geq
k-3}\int_0^t\|\wt{\D}^{\rm h}_{k'}\na_\h w_\Phi(t')\|_{L^\infty_{\rm
v}(L^2_{\rm h})}\|e^{\Psi}\D_{k'}^{\rm
h}w_\Phi(t')\|_{L^2_+}\|e^{\Psi}\D_{k}^{\rm
h}w_\Phi(t')\|_{L^2_+}\,dt', \end{split} \eeno yet observing that
\beno
\begin{split} \|\wt{\D}^{\rm h}_{k'}\na_\h
w_\Phi(t')\|_{L^\infty_{\rm v}(L^2_{\rm h})}\lesssim &
2^{k'}\|\wt{\D}^{\rm h}_{k'}w_\Phi(t')\|_{L^2_{\rm h}(L^\infty_{\rm
v})}\lesssim 2^{k'}\|\wt{\D}^{\rm h}_{k'}\p_yw_\Phi(t')\|_{L^2_{\rm
h}(L^1_{\rm v})}\\
 \lesssim & 2^{k'}\w{t'}^{\f14}\|e^\Psi\wt{\D}^{\rm
h}_{k'}\p_yw_\Phi(t')\|_{L^2_+} \lesssim 2^{\bigl(\f{3-d}2\bigr)k}\w{t'}^{\f14}\|\p_yw_\Phi(t')\|_{\cB^{\f{d-1}2,0}}\\
\lesssim& 2^{\bigl(\f{3-d}2\bigr)k} \dot{\tht}(t').
\end{split} \eeno
we find
 \beno
\begin{split}
\int_0^t\bigl|&\bigl(e^\Psi\D_k^{\rm h}\bigl[R^{\rm h}(w,\na_\h
w)\bigr]_\Phi\ |\ e^\Psi\D_k^{\rm
h}w_\Phi\bigr)_{L^2_+}\bigr|\,dt'\\
&\lesssim 2^{{\bigl(\f{d-1}2\bigr)k}}\sum_{k'\geq
k-3}2^{\bigl(\f{3-d}2\bigr)k'}\Bigl(\int_0^t\dot{\tht}(t')\|e^{\Psi}\D_{k'}^{\rm
h}w_\Phi(t')\|_{L^2_+}^2\,dt'\Bigr)^{\f12}\\
&\qquad\qquad\qquad\qquad\qquad\qquad\qquad\qquad\times\Bigl(\int_0^t\dot{\tht}(t')\|e^{\Psi}\D_{k}^{\rm
h}w_\Phi(t')\|_{L^2_+}^2\,dt'\Bigr)^{\f12}, \end{split} \eeno which
together with Definition \ref{def1.1} ensures that \beq\label{1.12b}
\begin{split}
\int_0^t\bigl|&\bigl(e^\Psi\D_k^{\rm h}\bigl[R^{\rm h}(w,\na_\h
w)\bigr]_\Phi\ |\ e^\Psi\D_k^{\rm
h}w_\Phi\bigr)_{L^2_+}\bigr|\,dt'\\
&\lesssim d_k2^{-\f{k}2}\Bigl( \sum_{k'\geq
k-3}d_{k'}2^{-\bigl(\f{2d-3}2\bigr){k'}}\Bigr)\|e^\Psi
w_\Phi\|_{\wt{L}^2_{t,\dot{\tht}(t)}(\cB^{\f{d}2,0})}^2\\
&\lesssim d_k^22^{-(d-1)k}\|e^\Psi
w_\Phi\|_{\wt{L}^2_{t,\dot{\tht}(t)}(\cB^{\f{d}2,0})}^2. \end{split}
 \eeq

In summary,  we arrive at \beq\label{1.13}
\int_0^t\bigl|\bigl(e^\Psi\D_k^{\rm h}[w\cdot\na_hw]_\Phi\ |\
e^\Psi\D_k^{\rm h}w_\Phi\bigr)_{L^2_+}\bigr|\,dt'\lesssim
d_k^22^{-(d-1)k}\|e^\Psi
w_\Phi\|_{\wt{L}^2_{t,\dot{\tht}(t)}(\cB^{\f{d}2,0})}^2. \eeq

\no $\bullet$ \underline{Estimate of $\int_0^t\bigl(e^\Psi\D_k^{\rm
h}\bigl[\int_0^y\dive_\h w\,dy'\p_yw\bigr]_\Phi\ |\ e^\Psi\D_k^{\rm
h}w_\Phi\bigr)_{L^2_+}\,dt'$}\vspace{0.2cm}

Once again we first get, by applying Bony's decomposition
\eqref{Bony} for $\int_0^y\dive_\h w\,dy'\p_yw$ in the horizontal
variable $x,$ that \beno \int_0^y\dive_\h w\,dy'\p_yw=T^{\rm
h}_{\int_0^y\dive_\h w\,dy'}\p_yw+T^{\rm h}_{\p_yw}\int_0^y\dive_\h
w\,dy'+R^{\rm h}\bigl(\int_0^y\dive_\h w\,dy',\p_yw\bigr).
 \eeno
It is easy to observe that \beno \begin{split}
\int_0^t\bigl|\bigl(&e^\Psi\D_k^{\rm h}\bigl[T_{\int_0^y\dive_\h
w\,dy'}\p_yw\bigr]_\Phi\ |\ e^\Psi\D_k^{\rm
h}w_\Phi\bigr)_{L^2_+}\,dt'\\ \lesssim &\sum_{|k'-k|\leq
4}\int_0^t\|S_{k'-1}^{\rm h}\bigl(\int_0^y\dive_\h
w_\Phi(t')\,dy'\bigr)\|_{L^\infty_+}\|e^{\Psi}\D_{k'}^{\rm
h}\p_yw_\Phi(t')\|_{L^2_+}\|e^{\Psi}\D_{k}^{\rm
h}w_\Phi(t')\|_{L^2_+}\,dt'\\
\lesssim &\sum_{|k'-k|\leq
4}2^{-\bigl(\f{d-1}2\bigr){k'}}\int_0^t\|S_{k'-1}^{\rm
h}\bigl(\int_0^y\dive_\h
w_\Phi(t')\,dy'\bigr)\|_{L^\infty_+}\\
&\qquad\qquad\qquad\qquad\qquad\qquad\qquad\qquad\quad\times\|e^{\Psi}\p_yw_\Phi(t')\|_{\cB^{\f{d-1}2,0}}\|e^{\Psi}\D_{k}^{\rm
h}w_\Phi(t')\|_{L^2_+}\,dt'. \end{split} \eeno Then due to
\eqref{1.9}, one has
 \beno \begin{split}
\int_0^t\bigl|\bigl(&e^\Psi\D_k^{\rm h}\bigl[T_{\int_0^y\dive_\h
w\,dy'}\p_yw\bigr]_\Phi\ |\ e^\Psi\D_k^{\rm
h}w_\Phi\bigr)_{L^2_+}\,dt'\\
\lesssim &\sum_{|k'-k|\leq
4}2^{-\bigl(\f{d-1}2\bigr){k'}}\Bigl(\int_0^t\w{t'}^{-\f12}\dot{\tht}(t')\|S_{k'-1}^{\rm
h}\bigl(\int_0^y\dive_\h
w_\Phi(t')\,dy'\bigr)\|_{L^\infty_+}^2\,dt'\Bigr)^{\f12}\\
&\qquad\qquad\qquad\qquad\qquad\qquad\qquad\qquad\quad\times\Bigl(\int_0^t\dot{\tht}(t')\|e^{\Psi}\D_{k}^{\rm
h}w_\Phi(t')\|_{L^2_+}^2\,dt'\Bigr)^{\f12}.
\end{split}
\eeno However in view of Definition \ref{def1.1}, by applying Lemma
\ref{lem:Bern}, one has \beno
\begin{split}
\Bigl(\int_0^t\w{t'}^{-\f12}\dot{\tht}(t')&\|S_{k'-1}^{\rm
h}\bigl(\int_0^y\dive_\h w_\Phi(t')\,dy'\bigr)\|_{L^\infty_+}^2\,dt'\Bigr)^{\f12}\\
\lesssim & \sum_{\ell\leq
k'-2}2^{\bigl(\f{d+1}2\bigr){\ell}}\Bigl(\int_0^t\w{t'}^{-\f12}\dot{\tht}(t')\|\D_{\ell}^{\rm
h}w_\Phi(t')\|_{L^1_{\rm v}(L^2_{\rm
h})}^2\,dt'\Bigr)^{\f12}\\\lesssim & \sum_{\ell\leq
k'-2}2^{\bigl(\f{d+1}2\bigr){\ell}}\Bigl(\int_0^t\dot{\tht}(t')\|e^{\Psi}\D_{\ell}^{\rm
h}w_\Phi(t')\|_{L^2_+}^2\,dt'\Bigr)^{\f12}\\
\lesssim & d_{k'}2^{\f{k'}2}\|e^\Psi
w_\Phi\|_{\wt{L}^2_{t,\dot{\tht}(t)}(\cB^{\f{d}2,0})}.
\end{split}
\eeno Hence we obtain \beno \int_0^t\bigl|\bigl(e^\Psi\D_k^{\rm
h}\bigl[T_{\int_0^y\dive_\h w\,dy'}\p_yw\bigr]_\Phi\ |\
e^\Psi\D_k^{\rm h}w_\Phi\bigr)_{L^2_+}\bigr|\,dt'\lesssim
d_{k}^22^{-(d-1)k}\|e^\Psi
w_\Phi\|_{\wt{L}^2_{t,\dot{\tht}(t)}(\cB^{\f{d}2,0})}^2. \eeno By
the same manner, we write \beno \begin{split}
\int_0^t\bigl|\bigl(&e^\Psi\D_k^{\rm
h}\bigl[T_{\p_yw}\int_0^y\dive_\h w\,dy'\bigr]_\Phi\ |\
e^\Psi\D_k^{\rm h}w_\Phi\bigr)_{L^2_+}\bigr|\,dt'\\ \lesssim
&\sum_{|k'-k|\leq 4}\int_0^t\|e^\Psi S_{k'-1}^{\rm
h}(\p_yw_\Phi(t'))\|_{L^2_{\rm v}(L^\infty_{\rm h})}\\
&\qquad\qquad\qquad\times \|\D_{k'}^{\rm h}\int_0^y\dive_\h
w_\Phi(t')\,dy'\|_{L^\infty_{\rm v}(L^2_{\rm
h})}\|e^{\Psi}\D_{k}^{\rm
h}w_\Phi(t')\|_{L^2_+}\,dt'\\
\lesssim &\sum_{|k'-k|\leq
4}2^{k'}\int_0^t\|e^\Psi\p_yw_\Phi(t')\|_{\cB^{\f{d-1}2,0}}\|\D_{k'}^{\rm
h}w_\Phi(t')\|_{L^1_{\rm v}(L^2_{\rm h})}\|e^{\Psi}\D_{k}^{\rm
h}w_\Phi(t')\|_{L^2_+}\,dt'\\
\lesssim &\sum_{|k'-k|\leq
4}2^{k'}\int_0^t\dot{\tht}(t')\|e^\Psi\D_{k'}^{\rm
h}w_\Phi(t')\|_{L^2_+}\|e^{\Psi}\D_{k}^{\rm
h}w_\Phi(t')\|_{L^2_+}\,dt',
\end{split}
\eeno which  along with a similar derivation of \eqref{1.12a} leads
to \beno \int_0^t\bigl|\bigl(e^\Psi\D_k^{\rm
h}\bigl[T_{\p_yw}\int_0^y\dive_\h  w\,dy'\bigr]_\Phi\ |\
e^\Psi\D_k^{\rm h}w_\Phi\bigr)_{L^2_+}\bigr|\,dt'\lesssim
d_{k}^22^{-(d-1)k}\|e^\Psi
w_\Phi\|_{\wt{L}^2_{t,\dot{\tht}(t)}(\cB^{\f{d}2,0})}^2. \eeno
Finally, it follows from Lemma \ref{lem:Bern}  that
\beno\begin{split} \int_0^t\bigl|\bigl(&e^\Psi\D_k^{\rm
h}\bigl[R^{\rm h}\bigl(\int_0^y\dive_\h
w\,dy',{\p_yw}\bigr)\bigr]_\Phi\ |\ e^\Psi\D_k^{\rm
h}w_\Phi\bigr)_{L^2_+}\bigr|\,dt'\\ \lesssim
&2^{\bigl(\f{d-1}2\bigr)k}\sum_{k'\geq k-3}\int_0^t\|{\D}_{k'}^{\rm
h}(\int_0^y\dive_\h w_\Phi(t')\,dy')\|_{L^\infty_{\rm v}(L^2_{\rm
h})}\\
&\qquad\qquad\qquad\qquad\times \|e^\Psi\wt{\D}_{k'}^{\rm
h}\p_yw_\Phi(t')\|_{L^2_+}\|e^{\Psi}\D_{k}^{\rm
h}w_\Phi(t')\|_{L^2_+}\,dt'
\end{split}
\eeno yet since \beno
\begin{split}
\|{\D}_{k'}^{\rm h}(\int_0^y\dive_\h
w_\Phi(t')\,dy')\|_{L^\infty_{\rm v}(L^2_{\rm h})}\lesssim  &
2^{k'}\|{\D}_{k'}^{\rm h} w_\Phi(t')\|_{L^1_{\rm v}(L^2_{\rm h})}\\
\lesssim & 2^{k'}\w{t'}^{\f14}\|e^{\Psi}{\D}_{k'}^{\rm h}
w_\Phi(t')\|_{L^2_+},
\end{split}
\eeno we deduce by a similar derivation of  \eqref{1.12b} that
\beno\begin{split} \int_0^t\bigl|\bigl(&e^\Psi\D_k^{\rm
h}\bigl[R^{\rm h}\bigl(\int_0^y\dive_\h
w\,dy',{\p_yw}\bigr)\bigr]_\Phi\ |\ e^\Psi\D_k^{\rm
h}w_\Phi\bigr)_{L^2_+}\bigr|\,dt'\\
\lesssim & 2^{\bigl(\f{d-1}2\bigr)k}\sum_{k'\geq k-
3}2^{\bigl(\f{3-d}2\bigr)k'}\int_0^t\w{t'}^{\f14}\|\p_yw_\Phi(t')\|_{\cB^{\f{d-1}2,0}}\|e^\Psi\D_{k'}^{\rm
h}w_\Phi(t')\|_{L^2_+}\|e^{\Psi}\D_{k}^{\rm
h}w_\Phi(t')\|_{L^2_+}\,dt'\\
\lesssim & 2^{\bigl(\f{d-1}2\bigr)k}\sum_{k'\geq k-
3}2^{\bigl(\f{3-d}2\bigr)k'}\Bigl(\int_0^t\dot{\theta}(t')\|e^\Psi\D_{k'}^{\rm
h}w_\Phi(t')\|_{L^2_+}^2\,dt'\Bigr)^{\f12}\Bigl(\int_0^t\dot{\theta}(t')\|e^{\Psi}\D_{k}^{\rm
h}w_\Phi(t')\|_{L^2_+}^2\,dt'\Bigr)^{\f12}\\
 \lesssim & d_{k}^22^{-(d-1)k}\|e^\Psi
w_\Phi\|_{\wt{L}^2_{t,\dot{\tht}(t)}(\cB^{\f{d}2,0})}^2.
\end{split}
\eeno

As a consequence, we achieve \beq\label{1.14}
\int_0^t\bigl|\bigl(e^\Psi\D_k^{\rm h}\bigl[\int_0^y\dive_\h
w\,dy'\p_yw\bigr]_\Phi\ |\ e^\Psi\D_k^{\rm
h}w_\Phi\bigr)_{L^2_+}\bigr|\,dt'\lesssim d_{k}^22^{-(d-1)k}\|e^\Psi
w_\Phi\|_{\wt{L}^2_{t,\dot{\tht}(t)}(\cB^{\f{d}2,0})}^2. \eeq

\no $\bullet$ \underline{Estimate of
$\int_0^t\bigl(e^\Psi\int_0^y\D_k^{\rm h}\dive_\h
w_\Phi\,dy'\p_yu^s\ |\ e^\Psi\D_k^{\rm
h}w_\Phi\bigr)_{L^2_+}\,dt'$}\vspace{0.2cm}

By virtue of Definition \ref{def1.1} and \eqref{1.9}, we get
\beno
\begin{split}
\int_0^t\bigl|\bigl(&e^\Psi\int_0^y\D_k^{\rm h}\dive_\h
w_\Phi\,dy'\p_yu^s\ |\ e^\Psi\D_k^{\rm
h}w_\Phi\bigr)_{L^2_+}\bigr|\,dt'\\
\lesssim &2^k\int_0^t\|e^{\Psi}\p_yu^s(t')\|_{L^2_{\rm v}}\|\D_k^{\rm h}w_\Phi(t')\|_{L^1_{\rm
v}(L^2_{\rm h})}\|e^\Psi\D_k^{\rm
h}w_\Phi(t')\|_{L^2_+}\,dt'\\
\lesssim & 2^k\int_0^t\dot{\tht}(t')\|e^\Psi\D_k^{\rm
h}w_\Phi(t')\|_{L^2_+}\|e^\Psi\D_k^{\rm h}w_\Phi(t')\|_{L^2_+}\,dt',
\end{split} \eeno
from which, a similar derivation of \eqref{1.12a} gives rise to \beq
\label{1.15} \int_0^t\bigl|\bigl(e^\Psi\int_0^y\p_x\D_k^{\rm
h}w_\Phi\,dy'\p_yu^s\ |\ e^\Psi\D_k^{\rm
h}w_\Phi\bigr)_{L^2_+}\bigr|\,dt'\lesssim d_{k}^22^{-(d-1)k}\|e^\Psi
w_\Phi\|_{\wt{L}^2_{t,\dot{\tht}(t)}(\cB^{\f{d}2,0})}^2. \eeq

\no $\bullet$ \underline{Estimate of
$-\int_0^t\bigl(e^\Psi\p_{yy}\D_k^{\rm h}w_\Phi\ |\ e^\Psi\D_k^{\rm
h}w_\Phi\bigr)_{L^2_+}\,dt'$}\vspace{0.2cm}

We get, by using integration by parts, that \beq\label{1.16}
\begin{split}
&-\int_0^t\bigl(e^\Psi\p_{yy}\D_k^{\rm h}w_\Phi\ |\ e^\Psi\D_k^{\rm
h}w_\Phi\bigr)_{L^2_+}\,dt'\\
&\quad=\|e^\Psi\D_k^{\rm
h}\p_yw_\Phi\|_{L^2_t(L^2_+)}^2+2\int_0^t\int_{\R^2_+}\pa_y\Psi e^{2\Psi}\D_k^{\rm
h}w_\Phi(t')\D_k^{\rm h}\p_yw_\Phi(t')\,dx\,dy\,dt'\\
&\quad\geq \f12\|e^\Psi\D_k^{\rm
h}\p_yw_\Phi\|_{L^2_t(L^2_+)}^2-2\int_0^t\int_{\R^d_+}(\pa_y\Psi)^2|e^\Psi\D_k^{\rm
h} w_\Phi(t')|^2\,dt'.
\end{split}
\eeq

Now we are in a position to complete the existence part of Theorem
\ref{th1.1}:\\

\begin{proof}[Proof of the existence part of Theorem \ref{th1.1}]
It is easy to observe that \beno \bigl(e^\Psi u^s\cdot\na_h\D_k^{\rm
h}w_\Phi\ |\ e^\Psi\D_k^{\rm h}w_\Phi\bigr)_{L^2_+}=0,\eeno and
\beno \lam\dot{\tht}(t)\bigl(e^\Psi|D_h|\D_k^{\rm h}w_\Phi\ |\
e^\Psi\D_k^{\rm h}w_\Phi\bigr)_{L^2_+}\geq
c\lam\dot{\tht}(t)2^k\|e^\Psi\D_{k}^{\rm h}w_\Phi(t)\|_{L^2_+}^2.
\eeno

Therefore in view of \eqref{eq2.8}, by integrating \eqref{1.8} over
$[0,t]$ and by resuming the Estimates \eqref{1.11}, \eqref{1.13},
\eqref{1.14}, \eqref{1.15} and \eqref{1.16} into the resulting
inequality, we conclude \beno
\begin{split} &\|e^\Psi\D_k^{\rm
h}w_\Phi\|_{L^\infty_t(L^2_+)}^2+c\lam
2^k\int_0^t\dot{\tht}(t')\|\D_k^{\rm
h}w_\Phi(t')\|_{L^2_+}^2\,dt'+\|e^\Psi\D_k^{\rm
h}\p_yw_\Phi\|_{L^2_t(L^2_+)}^2\\
&\qquad\qquad\qquad\qquad\qquad\leq \|e^{\f{1+y^2}8}\D_k^{\rm
h}e^{\delta|D|}w_0\|_{L^2_+}^2+Cd_k^22^{-(d-1)k} \|e^\Psi
w_\Phi\|_{\wt{L}^2_{t,\dot{\tht}(t)}(\cB^{\f{d}2,0})}^2.\end{split}
\eeno Taking square root of the above inequality and  multiplying
the resulting inequality by $2^{\bigl(\f{d-1}2\bigr)k}$ and summing
over $k\in\Z,$ we find for any $t\leq T^\ast$ \beq
\label{1.17}\begin{split} &\|e^\Psi
w_\Phi\|_{\wt{L}^\infty_t(\cB^{\f{d-1}2,0})}+c\sqrt{\lam}\|e^\Psi
w_\Phi\|_{\wt{L}^2_{t,\dot{\tht}(t)}(\cB^{\f{d}2,0})}+\|e^\Psi\p_yw_\Phi\|_{\wt{L}^2_t(\cB^{\f{d-1}2,0})}\\
&\qquad\qquad\qquad\qquad\qquad\leq \|e^{\f{1+y^2}8}
e^{\delta|D_x|}w_0\|_{\cB^{\f{d-1}2,0}}+\sqrt{C}\|e^\Psi
w_\Phi\|_{\wt{L}^2_{t,\dot{\tht}(t)}(\cB^{\f{d}2,0})}. \end{split}
\eeq Taking $\lam$ to be a large enough positive constant so that
$c^2\lam\geq C$ in \eqref{1.17} gives rise to \beq \label{1.18}
\|e^\Psi
w_\Phi\|_{\wt{L}^\infty_t(\cB^{\f{d-1}2,0})}+\|e^\Psi\p_yw_\Phi\|_{\wt{L}^2_t(\cB^{\f{d-1}2,0})}\leq
C\|e^{\f{1+y^2}8} e^{\delta|D|}w_0\|_{\cB^{\f{d-1}2,0}}. \eeq Hence
in view of \eqref{1.9}, we infer from Lemma \ref{lem:outflow} below
that \beq\label{1.19}
\begin{split}
\tht(t)\leq & \int_0^t\w{t'}^\f14\|e^\Psi
\p_yw_\Phi(t')\|_{\cB^{\f{d-1}2,0}}\,dt'+\int_0^t\w{t'}^\f14\|e^\Psi\p_yu^s(t')\|_{L^2_v}\,dt'\\
\leq &
C\w{t}^{\f34}\bigl(\|e^\Psi\p_yw_\Phi\|_{\wt{L}^2_t(\cB^{\f{d-1}2,0})}+\|e^\Psi\p_yu^s\|_{L^2_t(L^2_v)}\bigr)\\
\leq &C\w{t}^\f34\Bigl(\|e^{\f{1+y^2}8}
e^{\delta|D|}w_0\|_{\cB^{\f{d-1}2,0}}+\ve\Bigr).
\end{split}
\eeq In particular, under the assumption of \eqref{eq1.4},
\eqref{1.19} ensures that \beno \sup_{t\in
[0,\tau^\ast_\ve]}\tht(t)\leq \f{\de}{2\lam}\quad \textrm{for
}\tau^\ast_\ve\eqdefa \Bigl(\f{\delta}{4\lam C\ve}\Bigr)^{\f43}-1.
\eeno Therefore in view of \eqref{1.8a}, this ensures that
$T^\ast\geq \tau^\ast_\ve,$ and \eqref{1.18} implies
\eqref{f-class}. This completes the existence part of Theorem
\ref{th1.1}.
\end{proof}

It remains to prove the following lemma.

\begin{lemma}\label{lem:outflow}
{\sl Let $u^s$ be the global solution of (\ref{1.2}). Then one has
 \beq\label{eq3.15} \int_0^t\|e^\Psi\p_yu^s(t')\|_{L^2_v}^2\,dt'\le C\ve^2
 \eeq
for any $t\ge 0$.}
\end{lemma}

\begin{proof}\,Indeed, it is easy to observe from (\ref{1.2}) that  \beno u^s(t,y)=\f \ve
{2\sqrt{\pi}t}\int_0^\infty\Big(e^{-\f {(y-y')^2} {4t}}-e^{-\f
{(y+y')^2} {4t}}\Big)\chi(y')dy'{\bf e}. \eeno Due to the choice of
$\chi$, taking derivative with respect to $y$ gives
\begin{align*}
\pa_y u^s(t,y)&=\f \ve {2\sqrt{\pi}t}\int_0^\infty\Big(e^{-\f {(y-y')^2} {4t}}+e^{-\f {(y+y')^2} {4t}}\Big)\pa_y\chi(y')dy'{\bf e}\\
&=\f \ve {4\sqrt{\pi}t}\int_{-\infty}^{+\infty}\Big(e^{-\f {(y-y')^2} {4t}}+e^{-\f {(y+y')^2} {4t}}\Big)\pa_y\chi(y')dy'{\bf e}\\
&=\pa_y\f \ve {4\sqrt{\pi}t}\int_{-\infty}^{+\infty}\Big(e^{-\f
{(y-y')^2} {4t}}-e^{-\f {(y+y')^2} {4t}}\Big)\chi_1(y')dy'{\bf e}\\
&\eqdefa \f{\ve}2\p_y\bigl(u^s_+(t,y)+u^s_-(t,y)\bigr),
\end{align*}
for some $\chi_1(y)\in C_c^\infty(\R)$.

Notice that $u^s_\pm(t,y)$ verify
\begin{equation}\label{eq3.16}
 \quad\left\{\begin{array}{l}
\displaystyle \p_tu^s_\pm(t,y)-\p_{yy}u^s_\pm(t,y)=0\quad \mbox{in}\quad \R^+\times\R,\\
\displaystyle u^s_\pm |_{t=0}=\chi_1(\pm y),
\end{array}\right.
\end{equation}
Taking $L^2$ inner product of \eqref{eq3.16} with $e^{2\Psi}u^s_\pm$
and using integration by parts, we obtain \beno
\begin{split}
0=&\bigl(e^{\Psi}\p_tu^s_\pm\  |\ e^\Psi
u^s_\pm\bigr)_{L^2}-\bigl(e^{\Psi}\p_{yy}u^s_\pm(t,y)\  |\ e^\Psi
u^s_\pm\bigr)_{L^2}\\
=&\f12\f{d}{dt}\|e^{\Psi}u^s_\pm(t)\|_{L^2}^2+\|e^\Psi\p_yu^s_\pm(t)\|_{L^2}^2\\
&-\int_{\R}\p_t\Psi|e^\Psi u^s_\pm(t)|^2\,dy+2\int_{\R}\p_y\Psi
e^\Psi\p_yu^s_\pm e^\Psi u^s_\pm(t)\,dy,
\end{split}
\eeno which together \eqref{eq2.8} ensures that \beno
\begin{split}
\f12\Bigl(&\|e^{\Psi}u^s_\pm(t)\|_{L^2}^2+\|e^\Psi\p_yu^s_\pm\|_{L^2_t(L^2)}^2\Bigr)\\
\leq
&\|e^{\f{1+y^2}8}\chi_1(\pm
y)\|_{L^2}^2+\int_0^t\int_{\R}\bigl(\p_t\Psi-2(\p_y\Psi)^2\bigr)|e^\Psi
u^s_\pm(t')|^2\,dy\,dt'\\
\leq &\|e^{\f{1+y^2}8}\chi_1(\pm y)\|_{L^2}^2.
\end{split}
\eeno And hence \eqref{eq3.15} follows.
\end{proof}

\setcounter{equation}{0}
\section{The proof of the uniqueness part of Theorem \ref{th1.1}}

This section is devoted to the proof of the uniqueness part of
Theorem \ref{th1.1}. Let $w^1$ and $w^2$ be two solutions of
(\ref{1.3}) obtained by Theorem \ref{th1.1}. We denote $W\eqdefa
w^1-w^2$. Then in view of \eqref{1.3}, $W$ verifies
\begin{equation}\label{eq4.1}
 \quad\left\{\begin{array}{l}
\displaystyle \p_tW+u^s\cdot\na_\h W-\int_0^y\dive_\h\cdot W\,dy'\p_yu^s-\p_{yy}W+F=0,\\
\displaystyle W|_{y=0}=0,\qquad \lim_{y\to\infty} W=0, \\
\displaystyle W|_{t=0}=0,
\end{array}\right.
\end{equation}
where \beno F=w^1\cdot\na_\h W+W\cdot\na_hw^2-\int_0^y\dive_\h
w^{1}\,dy'\p_yW-\int_0^y\dive_\h W\,dy'\p_yw^2. \eeno Let
$\theta^i(t), i=1,2,$ be determined respectively by
\begin{equation}\label{eq4.2}
 \quad\left\{\begin{array}{l}
\displaystyle \dot{\tht^i}(t)=\w{t}^\f14\big(\|e^\Psi \p_yw_\Phi^i(t)\|_{\cB^{\f{d-1}2,0}}+\|e^{\Psi}\p_yu^s(t)\|_{L^2_{\rm v}}\big),\\
\displaystyle \tht^i|_{t=0}=0,
\end{array}\right.
\end{equation}
 and the phase function $\wt{\Phi}$ is defined by
 \beq\label{eq4.3} \begin{split}
 &\Theta(t)\eqdefa (\tht^1+\tht^2)(t),\quad
\wt{\Phi}(t,\xi)\eqdefa \Bigl(\f \de 2-\lam \Theta(t)\Bigr) |\xi| \andf\\
&\Phi^i(t,\xi)\eqdefa (\de-\lam \tht^i(t))|\xi|,\quad\mbox{for}\
i=1,2.
\end{split} \eeq
In what follows, we shall always take $t$ so small  that \beno \f
\de 2-\lam \Theta(t)\ge 0, \eeno so that there holds \beno
\wt{\Phi}(t,\xi)\leq
\wt{\Phi}(t,\xi-\eta)+\wt{\Phi}(t,\eta)\quad\mbox{for}\quad \forall\
\xi,\eta\in \R^{d-1}. \eeno

 Since $\wt{\Phi}(t)\leq
\min\bigl(\Phi_1(t), \Phi_2(t)\bigr),$ we have for $i=1,2,$
\beq\label{eq4.4}
\begin{split}
&\|e^{\Psi}w^i_{\wt{\Phi}}\|_{\wt{L}^\infty_t(\cB^{\f{d-1}2,0})}\leq
\|e^{\Psi}w^i_{\Phi^i}\|_{\wt{L}^\infty_t(\cB^{\f{d-1}2,0})}\andf\\
&\|e^{\Psi}\p_yw^i_{\wt{\Phi}}\|_{\wt{L}^2_t(\cB^{\f{d-1}2,0})}\leq
\|e^{\Psi}\p_yw^i_{\Phi^i}\|_{\wt{L}^2_t(\cB^{\f{d-1}2,0})}.
\end{split} \eeq
Furthermore, it follows from Definition \ref{def2.2} and
\eqref{eq2.4} that \beno
\|e^{\Psi}w^i_{\wt{\Phi}}\|_{\wt{L}^\infty_t(\cB^{\f{d+1}2,0})}=\sum_{k\in\Z}2^{\bigl(\f{k+1}2\bigr)k}\|e^\Psi\D_k^\h
w^i_{\wt{\Phi}}\|_{L^\infty_t(L^2_+)}, \eeno and for each time $t,$
one has \beno
\begin{split}
2^k\|e^\Psi\D_k^\h w^i_{\wt{\Phi}}(t)\|_{L^2_+}^2=
&2^k\int_{0}^\infty\int_{\R^{d-1}}e^{2\Psi}\varphi^2(2^{-k}|\xi|)e^{2\wt{\Phi}(t,\xi)}|\widehat{w^i}(t,\xi,y)|^2\,d\xi\,dy\\
\leq
&2\int_{0}^\infty\int_{\R^{d-1}}e^{2\Psi}\varphi^2(2^{-k}|\xi|)e^{2\Phi^i(t,\xi)}|\xi|e^{-\f{\de}2|\xi|}|\widehat{w^i}(t,\xi,y)|^2\,d\xi\,dy\\
\leq
&C_\de\int_{0}^\infty\int_{\R^{d-1}}e^{2\Psi}\varphi^2(2^{-k}|\xi|)e^{2\Phi^i(t,\xi)}|\widehat{w^i}(t,\xi,y)|^2\,d\xi\,dy\\
=& C_\de \|e^\Psi \D_\h^k w^i_{\Phi^i}(t)\|_{L^2_+}^2,
\end{split}
\eeno which yields \beq\label{eq4.4a}
\|e^{\Psi}w^i_{\wt{\Phi}}\|_{\wt{L}^\infty_t(\cB^{\f{d+1}2,0})}\leq
C_\de \|e^{\Psi}w^i_{\Phi^i}\|_{\wt{L}^\infty_t(\cB^{\f{d-1}2,0})}.
\eeq In fact, this inequality \eqref{eq4.4a} motivates us to
introduce the phase function $\wt{\Phi}(t,\xi)$ in \eqref{eq4.3}.

For $\wt{\Phi}(t,\xi)$ given by \eqref{eq4.3}, in view of
\eqref{eq4.1}, one has \beno \p_tW_{\wt{\Phi}}+\lam
\dot{\Theta}(t)|D_h|W_{\wt{\Phi}}+u^s\cdot\na_\h
W_{\wt{\Phi}}-\int_0^y\dive_\h
W_{\wt{\Phi}}\,dy'\p_yu^s-\p_{yy}W_{\wt{\Phi}}+F_{\wt{\Phi}}=0.
\eeno Then for $\Psi(t,y)$ given by \eqref{eq2.6}, by applying
$\D_k^\h$ to the above equation and then taking $L^2$ inner product
of the resulting equation with $e^{2\Psi}\D_k^\h W_{\wt{\Phi}},$ we
obtain \beno
\begin{split}
\bigl(&e^{\Psi}\D_k^{\rm h}\p_tW_{\widetilde{\Phi}}\ |\
e^{\Psi}\D_k^{\rm h}W_{\widetilde{\Phi}}\bigr)_{L^2_+}+\lam
\dot{\Theta} \bigl(e^{\Psi}|D_\h|\D_k^\h W_{\wt{\Phi}}\ |\ \D_k^\h
W_{\wt{\Phi}}\bigr)_{L^2_+}\\
&+\bigl(e^{\Psi}u^s\cdot\na_\h\D_k^\h W_{\wt{\Phi}}\ |\
e^{\Psi}\D_k^\h W_{\wt{\Phi}}\bigr)_{L^2_+}-\bigl(e^{\Psi}\D_k^\h
\p_{yy}W_{\wt{\Phi}}\ |\ e^{\Psi}\D_k^\h W_{\wt{\Phi}}\bigr)_{L^2_+}\\
&-\bigl(e^{\Psi}\int_0^y\D_k^\h\dive_\h W_{\wt{\Phi}}\,dy'\p_yu^s\
|\ e^{\Psi}\D_k^\h W_{\wt{\Phi}}\bigr)_{L^2_+}
+\bigl(e^{\Psi}\D_k^\h F_{\wt{\Phi}} \ |\
e^{\Psi}\D_k^\h W_{\wt{\Phi}}\bigr)_{L^2_+},
\end{split}
\eeno from which and similar derivations of \eqref{1.11},
\eqref{1.16}, we  infer \beq\label{eq4.5}
\begin{split}
&\|e^{\Psi}\D_k^{\rm
h}W_{\widetilde{\Phi}}\|_{L^\infty_t(L^2_+)}^2+c\lambda
2^k\int_0^t\dot{\Theta}(t')\|e^{\Psi}\D_k^{\rm
h}W_{\widetilde{\Phi}}(t')\|_{L^2_+}^2\,dt'+
\|e^{\Psi}\D_k^{\rm h}\pa_yW_{\widetilde{\Phi}}\|_{L^2_t(L^2_+)}^2\\
&\quad\le \int_0^t\Bigl(\bigl(e^{\Psi}\int_0^y\D_k^\h\dive_\h
W_{\wt{\Phi}}\,dy'\p_yu^s\ |\ e^{\Psi}\D_k^\h
W_{\wt{\Phi}}\bigr)_{L^2_+}\\
&\qquad\qquad\qquad\qquad\qquad\qquad\qquad -\bigl(e^{\Psi}\D_k^{\rm
h} F_{\wt{\Phi}}\ |\ e^{\Psi}\D_k^{\rm
h}W_{\widetilde{\Phi}}\bigr)_{L^2_+}\Bigr)(t')\,dt'.
\end{split}
\eeq

In what follows, we handle term by term above.\vspace{0.2cm}

\no $\bullet$ \underline{Estimate of $\int_0^t\bigl(e^{\Psi}\D_k^\h
(w^1\cdot\na_\h W)_{\wt{\Phi}}\ |\ e^{\Psi}\D_k^\h
W_{\wt{\Phi}}\bigr)_{L^2_+}\,dt'$}\vspace{0.2cm}

We first get, by using Bony's decomposition \eqref{Bony} for
$w^1\cdot\na_\h W$ for the horizontal variables, that \beno
w^1\cdot\na_\h W=T^\h_{w^1}\na_\h W+T^\h_{\na_\h W}w^1+R^\h(w^1,
\na_\h W).
 \eeno
By using a similar derivation of \eqref{1.12} and \eqref{1.12a},  we
find \beno \|e^\Psi\D_k^\h [T^\h_{w^1}\na_\h
W]_{\wt{\Phi}}(t)\|_{L^2_+}\lesssim 2^k\sum_{|k'-k|\leq
4}\w{t}^{\f14}\|e^\Psi\p_yw^1_{\wt{\Phi}}(t)\|_{\cB^{\f{d-1}2,0}}\|e^\Psi\D_{k'}^\h
W_{\wt{\Phi}}(t)\|_{L^2_+},
 \eeno
which together with \eqref{eq4.3} and \eqref{eq4.4} implies that
\ben\label{eq:LH}
\begin{split}
\int_0^t\bigl|&\bigl(e^{\Psi}\D_k^\h [T^\h_{w^1}\na_\h
W]_{\wt{\Phi}}\
|\ e^{\Psi}\D_k^\h W_{\wt{\Phi}}\bigr)_{L^2_+}\bigr|\,dt'\\
\lesssim &2^{k}\sum_{|k'-k|\leq
4}\int_0^t\dot{\Theta}(t')\|e^\Psi\D_{k'}^\h
W_{\wt{\Phi}}(t')\|_{L^2_+}\|e^\Psi\D_{k}^\h
W_{\wt{\Phi}}(t')\|_{L^2_+}\,dt'\\
\lesssim & 2^{k}\sum_{|k'-k|\leq
4}\Bigl(\int_0^t\dot{\Theta}(t')\|e^\Psi\D_{k'}^\h
W_{\wt{\Phi}}(t')\|_{L^2_+}^2\,dt'\Bigr)^{\f12}\Bigl(\int_0^t\dot{\Theta}(t')\|e^\Psi\D_{k}^\h
W_{\wt{\Phi}}(t')\|_{L^2_+}^2\,dt'\Bigr)^{\f12}\\
\lesssim &
d_k^22^{-(d-1)k}\|e^\Psi W_{\wt{\Phi}}\|_{\wt{L}^2_{t,\dot{\Theta}}(\cB^{\f{d}2,0})}^2.
\end{split}
\een

Note that according to \eqref{eq4.2} and \eqref{eq4.3}, we get, by a
similar derivation of \eqref{1.12}, that \beq\label{eq4.6}
\begin{split}
\|\D_k^\h w^1_{\wt{\Phi}}(t)\|_{L^\infty_{\rm v}(L^2_\h)}\lesssim&
\w{t}^{\f14}\|e^\Psi\D_k^\h\p_yw^1_{\wt{\Phi}}(t)\|_{L^2_+}\\
\lesssim &
2^{-\bigl(\f{d-1}2\bigr)k}\w{t}^{\f14}\|e^\Psi\p_yw^1_{\wt{\Phi}}(t)\|_{\cB^{\f{d-1}2,0}}\lesssim
2^{-\bigl(\f{d-1}2\bigr)k}\dot{\Theta}(t),
\end{split}
\eeq while it follows from Lemma \ref{lem:Bern}  that \beno
\begin{split}
\Bigl(\int_0^t\dot{\Theta}(t')&\|e^\Psi S_{k'-1}^\h\na_\h
W_{\wt{\Phi}}(t')\|_{L^2_{\rm v}(L^\infty_\h)}^2\,dt'\Bigr)^{\f12}\\
\lesssim &\sum_{\ell\leq
k'-2}2^{\bigl(\f{d+1}2\bigr)k'}\Bigl(\int_0^t
\dot{\Theta}(t')\|e^\Psi\D_\ell^\h
W_{\wt{\Phi}}(t')\|_{L^2_+}^2\,dt'\Bigr)^{\f12}\\
\lesssim &
d_k2^{\f{k'}2}\|e^\Psi W_{\wt{\Phi}}\|_{\wt{L}^2_{t,\dot{\Theta}}(\cB^{\f{d}2,0})}.
\end{split}
\eeno
Hence we obtain
\beno
\begin{split}
\int_0^t\bigl|&\bigl(e^{\Psi}\D_k^\h [T^\h_{\na_\h
W}w^1]_{\wt{\Phi}}\ |\ e^{\Psi}\D_k^\h
W_{\wt{\Phi}}\bigr)_{L^2_+}\bigr|\,dt'\\
\lesssim & \sum_{|k'-k|\leq 4}\int_0^t\|e^\Psi S_{k'-1}^\h\na_\h
W_{\wt{\Phi}}(t')\|_{L^2_{\rm v}(L^\infty_\h)}\|\D_{k'}^\h
w^1_{\wt{\Phi}}(t')\|_{L^\infty_{\rm v}(L^2_\h)}\|e^\Psi\D_k^\h
W_{\wt{\Phi}}(t')\|_{L^2_+}\,dt'\\
\lesssim &\sum_{|k'-k|\leq
4}2^{\bigl(\f{1-d}2\bigr)k'}\int_0^t\dot{\Theta}(t')\|e^\Psi
S_{k'-1}^\h\na_\h W_{\wt{\Phi}}(t')\|_{L^2_{\rm
v}(L^\infty_\h)}\|e^\Psi\D_k^\h W_{\wt{\Phi}}(t')\|_{L^2_+}\,dt'\\
\lesssim &\sum_{|k'-k|\leq
4}2^{\bigl(\f{1-d}2\bigr)k'}\Bigl(\int_0^t\dot{\Theta}(t')\|e^\Psi
S_{k'-1}^\h\na_\h W_{\wt{\Phi}}(t')\|_{L^2_{\rm
v}(L^\infty_\h)}^2\,dt'\Bigr)^{\f12}\\
&\qquad\qquad\qquad\qquad\qquad\qquad\qquad\times \Bigl(\int_0^t\dot{\Theta}(t')\|e^\Psi\D_k^\h W_{\wt{\Phi}}(t')\|_{L^2_+}\,dt'\Bigr)^{\f12}\\
\lesssim &
d_k^22^{-(d-1)k}\|e^\Psi W_{\wt{\Phi}}\|_{\wt{L}^2_{t,\dot{\Theta}}(\cB^{\f{d}2,0})}^2.
\end{split}
\eeno

While by applying Lemma \ref{lem:Bern} and \eqref{eq4.6}, we write
 \beno
\begin{split}
\int_0^t\bigl|&\bigl(e^{\Psi}\D_k^\h [R^\h({w^1},\na_\h
W)]_{\wt{\Phi}}\
|\ e^{\Psi}\D_k^\h W_{\wt{\Phi}}\bigr)_{L^2_+}\bigr|\,dt'\\
\lesssim &2^{\bigl(\f{d-1}2\bigr)k}\sum_{k'\geq
k-3}\int_0^t\|e^\Psi\D_{k'}^\h\na_\h
W_{\wt{\Phi}}(t')\|_{L^2_+}\|\wt{\D}_{k'}^\h
w^1_{\wt{\Phi}}(t')\|_{L^\infty_{\rm v}(L^2_\h)}\|e^\Psi\D_{k}^\h
W_{\wt{\Phi}}(t')\|_{L^2_+}\,dt' \\
\lesssim &2^{\bigl(\f{d-1}2\bigr)k}\sum_{k'\geq
k-3}2^{\bigl(\f{3-d}2\bigr)k'}\int_0^t\dot{\Theta}(t')\|e^\Psi\D_{k'}^\h
W_{\wt{\Phi}}(t')\|_{L^2_+}\|e^\Psi\D_{k}^\h
W_{\wt{\Phi}}(t')\|_{L^2_+}\,dt',
\end{split} \eeno
which together with Definition \ref{def1.1} implies
\ben\label{eq:HH}
\begin{split}
\int_0^t\bigl|\bigl(e^{\Psi}\D_k^\h [R^\h({w^1},&\na_\h
W)]_{\wt{\Phi}}\ |\ e^{\Psi}\D_k^\h
W_{\wt{\Phi}}\bigr)_{L^2_+}\bigr|\,dt'\\
 \lesssim &
d_k2^{-\f{k}2}\Bigl(\sum_{k'\geq
k-3}d_{k'}2^{\bigl(\f{3-2d}2\bigr)k'}\Bigr)\| e^\Psi W_{\wt{\Phi}}\|_{\wt{L}^2_{t,\dot{\Theta}}(\cB^{\f{d}2,0})}^2\\
\lesssim &
d_k^22^{-(d-1)k}\|e^\Psi W_{\wt{\Phi}}\|_{\wt{L}^2_{t,\dot{\Theta}}(\cB^{\f{d}2,0})}^2.
\end{split}
\een

 Therefore, we obtain
 \beq\label{eq4.7}
 \int_0^t\bigl|\bigl(e^{\Psi}\D_k^\h
(w^1\cdot\na_\h W)_{\wt{\Phi}}\ |\ e^{\Psi}\D_k^\h
W_{\wt{\Phi}}\bigr)_{L^2_+}\bigr|\,dt'\lesssim
d_k^22^{-(d-1)k}\|e^\Psi W_{\wt{\Phi}}\|_{\wt{L}^2_{t,\dot{\Theta}}(\cB^{\f{d}2,0})}^2.
\eeq

\no $\bullet$ \underline{Estimate of $\int_0^t\bigl(e^{\Psi}\D_k^\h
(W\cdot\na_\h w^2)_{\wt{\Phi}}\ |\ e^{\Psi}\D_k^\h
W_{\wt{\Phi}}\bigr)_{L^2_+}\,dt'$}\vspace{0.2cm}

By using Bony's decomposition \eqref{Bony} for $W\cdot\na_\h w^2$
for the horizontal variables, we get \beno W\cdot\na_\h
w^2=T^\h_{W}\na_\h w^2+T^\h_{\na_\h w^2}W+R^\h(W, \na_\h w^2). \eeno

Along the same line to the derivation of  (\ref{eq:LH}) and
(\ref{eq:HH}), we find \beno \int_0^t\bigl|\bigl(e^{\Psi}\D_k^\h
[T^\h_{\na_h w^2}W]_{\wt{\Phi}}\ |\ e^{\Psi}\D_k^\h
W_{\wt{\Phi}}\bigr)_{L^2_+}\bigr|\,dt'\lesssim
d_k^22^{-(d-1)k}\|e^\Psi
W_{\wt{\Phi}}\|_{\wt{L}^2_{t,\dot{\Theta}}(\cB^{\f{d}2,0})}^2, \eeno
and \beno \int_0^t\bigl|\bigl(e^{\Psi}\D_k^\h [R^\h({W},\na_\h
w^2)]_{\wt{\Phi}}\ |\ e^{\Psi}\D_k^\h
W_{\wt{\Phi}}\bigr)_{L^2_+}\bigr|\,dt'\lesssim
d_k^22^{-(d-1)k}\|e^\Psi
W_{\wt{\Phi}}\|_{\wt{L}^2_{t,\dot{\Theta}}(\cB^{\f{d}2,0})}^2. \eeno
Whereas applying Lemma \ref{lem:Bern} yields \beno
\begin{split}
\|e^\Psi\D_k^\h [T^\h_{W}\na_\h
w^2]_{\wt{\Phi}}(t)\|_{L^2_+}\lesssim& 2^k\sum_{|k'-k|\leq
4}\|S^h_{k'-1}
W_{\wt{\Phi}}(t)\|_{L^\infty_+}\|e^\Psi \Delta_{k'}^hw^2_{\wt{\Phi}}\|_{L^\infty_t(L^2_+)}\\
\lesssim&2^{-\bigl(\f {d-1}2\bigr) k}d_k\w{t}^{\f14}\|e^\Psi\pa_y
W_{\wt{\Phi}}(t)\|_{\cB^{\f{d-1}2,0}}\|e^\Psi
w^2_{\wt{\Phi}}\|_{\wt{L}^\infty_t(\cB^{\f{d+1}2,0})},
\end{split}
\eeno from which, we infer \beq\label{eq4.10}
\begin{split}
\int_0^t&\bigl|\bigl(e^{\Psi}\D_k^\h [T^\h_{W}\na_h
w^2]_{\wt{\Phi}}\ |\ e^{\Psi}\D_k^\h
W_{\wt{\Phi}}\bigr)_{L^2_+}\bigr|\,dt'
\lesssim d_k^22^{-(d-1)k}\bigl(\w{t}^{\f32}-1\bigr)^{\f12}\\
&\qquad\qquad\times\|e^\Psi
\pa_yW_{\wt{\Phi}}\|_{\wt{L}^2_{t}(\cB^{\f{d-1}2,0})}\|e^\Psi
W_{\wt{\Phi}}\|_{\wt{L}^\infty_{t}(\cB^{\f{d-1}2,0})} \|e^\Psi
w^2_{\wt{\Phi}}(t)\|_{\wt{L}^\infty_t(\cB^{\f{d+1}2,0})}.
\end{split}
\eeq Therefore, in view \eqref{eq4.4a}, we obtain \beq\label{eq4.11}
\begin{split}
 \int_0^t\bigl|\bigl(&e^{\Psi}\D_k^\h
(W\cdot\na_\h w^2)_{\wt{\Phi}}\ |\ e^{\Psi}\D_k^\h
W_{\wt{\Phi}}\bigr)_{L^2_+}\bigr|\,dt'\lesssim d_k^22^{-(d-1)k}\Big(\|e^\Psi W_{\wt{\Phi}}\|_{\wt{L}^2_{t,\dot{\Theta}}(\cB^{\f{d}2,0})}^2\\
&\qquad\qquad+t^\f12\|e^\Psi
w^2_{\Phi^2}\|_{\wt{L}^\infty_t(\cB^{\f{d-1}2,0})}\|e^\Psi
\pa_yW_{\wt{\Phi}}\|_{\wt{L}^2_{t}(\cB^{\f{d-1}2,0})}\|e^\Psi
W_{\wt{\Phi}}\|_{\wt{L}^\infty_{t}(\cB^{\f{d-1}2,0})} \Big).
\end{split}
\eeq

\no $\bullet$ \underline{Estimate of $\int_0^t\bigl(e^{\Psi}\D_k^\h
(\int_0^y\dive_\h w^{1}\,dy'\p_yW)_{\wt{\Phi}}\ |\ e^{\Psi}\D_k^\h
W_{\wt{\Phi}}\bigr)_{L^2_+}\,dt'$}\vspace{0.2cm}

By applying Bony's decomposition and a similar trick of
\eqref{1.12}, we write \beno
\begin{split}
\|e^{\Psi}\D_k^\h &(\int_0^y\dive_\h
w^{1}\,dy'\p_yW)_{\wt{\Phi}}\|_{L^1_t(L^2_+)}\\
\lesssim & \sum_{k'\geq k-N_0}\int_0^t\Bigl(\|S_{k'-1}^\h\dive_\h
w^1_{\wt{\Phi}}(t')\|_{L^1_{\rm
v}(L^\infty_\h)}\|e^\Psi\D_{k'}^\h\p_yW_{\wt{\Phi}}(t')\|_{L^2_+}\\
&\qquad\qquad+\|\D_{k'}^\h\dive_\h w^1_{\wt{\Phi}}(t')\|_{L^1_{\rm
v}(L^2_\h)}\|e^\Psi
S_{k'+2}^\h\p_yW_{\wt{\Phi}}(t')\|_{L^2_{\rm v}(L^\infty_\h)}\Bigr)\,dt'\\
\lesssim & \sum_{k'\geq
k-N_0}\|\w{t'}^{\f14}\|_{L^2_t}\Bigl(\|e^\Psi S_{k'-1}^\h\dive_\h
w^1_{\wt{\Phi}}\|_{L^\infty_t(L^2_{\rm
v}(L^\infty_\h))}\|e^\Psi\D_{k'}^\h\p_yW_{\wt{\Phi}}\|_{L^2_t(L^2_+)}\\
&\qquad\qquad\qquad\qquad+2^k\|e^\Psi\D_{k'}^\h
w^1_{\wt{\Phi}}\|_{L^\infty_t(L^2_+)}\|e^\Psi
S_{k'+2}^\h\p_yW_{\wt{\Phi}}\|_{L^2_t(L^2_{\rm
v}(L^\infty_\h))}\Bigr),
\end{split}
\eeno from which and Lemma \ref{lem:Bern}, we infer \beno
\begin{split}
\|e^{\Psi}\D_k^\h &(\int_0^y\dive_\h
w^{1}\,dy'\p_yW)_{\wt{\Phi}}\|_{L^1_t(L^2_+)}\\
\lesssim & \bigl(\w{t}^{\f32}-1\bigr)^{\f12}\sum_{k'\geq
k-N_0}d_{k'}2^{-\bigl(\f{d-1}2\bigr)k'}\|w^1_{\wt{\Phi}}\|_{\wt{L}^\infty_t(\cB^{\f{d+1}2,0})}\|\p_yW_{\wt{\Phi}}\|_{\wt{L}^2_t(\cB^{\f{d-1}2,0})}\\
\lesssim & t^{\f12}d_{k}2^{-\bigl(\f{d-1}2\bigr)k}\|e^\Psi
w^1_{\wt{\Phi}}\|_{\wt{L}^\infty_t(\cB^{\f{d+1}2,0})}\|e^\Psi
\p_yW_{\wt{\Phi}}\|_{\wt{L}^2_t(\cB^{\f{d-1}2,0})}.
\end{split}
\eeno
So that by virtue of \eqref{eq4.4a}, we obtain
\beq\label{eq4.12}
\begin{split}
&\int_0^t\bigl|\bigl(e^{\Psi}\D_k^\h (\int_0^y\dive_\h
w^{1}\,dy'\p_yW)_{\wt{\Phi}}\ |\ e^{\Psi}\D_k^\h
W_{\wt{\Phi}}\bigr)_{L^2_+}\bigr|\,dt'\\
&\qquad\lesssim t^{\f12} d_k^22^{-(d-1)k}\|e^\Psi
w^1_{\Phi^{1}}\|_{\wt{L}^\infty_t(\cB^{\f{d-1}2,0})}\|e^\Psi
\pa_yW_{\wt{\Phi}}\|_{\wt{L}^2_{t}(\cB^{\f{d-1}2,0})}\|e^\Psi
W_{\wt{\Phi}}\|_{\wt{L}^\infty_{t}(\cB^{\f{d-1}2,0})}.
\end{split}
\eeq

\no $\bullet$ \underline{Estimate of $\int_0^t\bigl(e^{\Psi}\D_k^\h
(\int_0^y\dive_\h W\,dy'\p_yw^2)_{\wt{\Phi}}\ |\ e^{\Psi}\D_k^\h
W_{\wt{\Phi}}\bigr)_{L^2_+}\,dt'$}\vspace{0.2cm}

The estimate of this term is the almost same as that of
(\ref{eq4.7}). Indeed applying Bony's decomposition \eqref{Bony}
gives \beno
\begin{split}
\int_0^t\bigl|&\bigl(e^{\Psi}\D_k^\h (\int_0^y\dive_\h
W\,dy'\p_yw^2)_{\wt{\Phi}}\ |\ e^{\Psi}\D_k^\h
W_{\wt{\Phi}}\bigr)_{L^2_+}\bigr|\,dt'\\
\lesssim &\sum_{k'\geq k-N_0}\int_0^t\Bigl(\|S_{k'-1}^\h \dive_\h
W_{\wt{\Phi}}(t')\|_{L^1_{\rm v}(L^\infty_\h)}\|e^\Psi
\D_{k'}^\h\p_yw^2_{\wt{\Phi}}(t')\|_{L^2_+}\\
&\qquad\qquad+\|\D_{k'}^\h \dive_\h W_{\wt{\Phi}}(t')\|_{L^1_{\rm
v}(L^2_\h)}\|e^\Psi S_{k'+2}^\h\p_yw^2_{\wt{\Phi}}(t')\|_{L^2_{\rm
v}(L^\infty_\h)}\Bigr)\|e^\Psi\D_k^\h
W_{\wt{\Phi}}(t')\|_{L^2_+}\,dt'\\
\lesssim &\sum_{k'\geq
k-N_0}2^{-\bigl(\f{d-1}2\bigr)k'}\int_0^t\w{t'}^{\f14}\|e^\Psi\p_yw^2_{\wt{\Phi}}(t)\|_{\cB^{\f{d-1}2,0}}
\bigl(\|e^\Psi S_{k'-1}^\h \dive_\h W_{\wt{\Phi}}(t')\|_{L^2_{\rm
v}(L^\infty_\h)}\\
&\qquad\qquad\qquad\qquad\qquad\qquad\qquad\qquad+2^{k'} \|e^\Psi
\D_{k'}^\h W_{\wt{\Phi}}(t')\|_{L^2_+}\bigr)\|e^\Psi\D_k^\h
W_{\wt{\Phi}}(t')\|_{L^2_+}\,dt',
\end{split}
\eeno from which and a similar derivation of \eqref{eq4.7}, we
arrive at
 \beq \label{eq4.14} \int_0^t\bigl(e^{\Psi}\D_k^\h
(\int_0^y\dive_\h W\,dy'\p_yw^2)_{\wt{\Phi}}\ |\ e^{\Psi}\D_k^\h
W_{\wt{\Phi}}\bigr)_{L^2_+}\,dt'\lesssim d_k^22^{-(d-1)k}\|e^\Psi
W_{\wt{\Phi}}\|_{\wt{L}^2_{t,\dot{\Theta}}(\cB^{\f{d}2,0})}^2. \eeq

Finally in view of \eqref{eq4.2} and \eqref{eq4.3}, one has \beno
\begin{split}
\int_0^t\bigl|\bigl(e^{\Psi}\D_k^\h (\int_0^y\dive_\h&
W\,dy')_{\wt{\Phi}}\p_y u^s\ |\ e^{\Psi}\D_k^\h
W_{\wt{\Phi}}\bigr)_{L^2_+}\bigr|\,dt'\\
\lesssim &2^k\int_0^t\w{t'}^{\f14}\|e^\Psi\p_yu^s(t')\|_{L^2_{\rm
v}}\|e^\Psi\D_k^\h W_{\wt{\Phi}}(t')\|_{L^2_+}^2\,dt',
\end{split}
\eeno from which and Definitions \ref{def2.2} and \ref{def1.1}, we
infer \beq\label{eq4.15} \int_0^t\bigl|\bigl(e^{\Psi}\D_k^\h
(\int_0^y\dive_\h W\,dy')_{\wt{\Phi}}\p_y u^s\ |\ e^{\Psi}\D_k^\h
W_{\wt{\Phi}}\bigr)_{L^2_+}\bigr|\,dt'\lesssim
d_k^22^{-(d-1)k}\|e^\Psi
W_{\wt{\Phi}}\|_{\wt{L}^2_{t,\dot{\Theta}}(\cB^{\f{d}2,0})}^2. \eeq

Now let us complete the proof of Theorem \ref{th1.1}.

\begin{proof}[Proof of the uniqueness part of Theorem \ref{th1.1}]
Resuming the Estimates (\ref{eq4.7}) and
(\ref{eq4.11})--(\ref{eq4.15}) into (\ref{eq4.5}), then taking
square root of the resulting inequality and  multiplying it by
$2^{\bigl(\f{d-1}2\bigr)k}$, we thus obtain, by  summing over the
final inequality for  $k\in\Z,$ that \beno
\begin{split}
\|e^{\Psi}&W_{\widetilde{\Phi}}\|_{\wt{L}^\infty_t(\cB^{\f {d-1}
2,0})}+c\sqrt{\lambda}\|e^{\Psi}W_{\widetilde{\Phi}}\|_{\wt{L}^2_{t,
\dot \Theta}(\cB^{\f d2,0})}
+\|e^{\Psi}\pa_yW_{\widetilde{\Phi}}\|_{\wt{L}^2_t(\cB^{\f {d-1} 2,0})}\\
&\le C\|e^{\Psi}W_{\widetilde{\Phi}}\|_{\wt{L}^2_{t, \dot
\Theta}(\cB^{\f d2,0})}+
\f12\|e^{\Psi}\pa_yW_{\widetilde{\Phi}}\|_{\wt{L}^2_t(\cB^{\f {d-1} 2,0})}\\
&\qquad+Ct^\f12\Big(\|e^{\Psi}w^1_{\Phi^1}\|_{\widetilde{L}^\infty_t(\cB^{\f
{d-1} 2,0})}
+\|e^{\Psi}w^2_{\Phi^{2}}\|_{\widetilde{L}^\infty_t(\cB^{\f {d-1}
2,0})}\Big)^\f12\|e^{\Psi}W_{\widetilde{\Phi}}\|_{\wt{L}^\infty_t(\cB^{\f
{d-1} 2,0})}.  \end{split}\eeno Hence, by taking $\lambda$ large
enough and $t$ being sufficiently small in the above inequality, we
deduce that
$\|e^{{\Psi}}W_{\widetilde{\Phi}}\|_{\wt{L}^\infty_t(\cB^{\f {d-1}
2,0})}=0$ for some small time $t$. The uniqueness for whole time of
existence can be deduced by a continuous argument. This completes
the proof of Theorem \ref{th1.1}.
\end{proof}

\section*{Acknowledgments}

Part of this work was done when we were visiting Morningside Center of the Academy of Mathematics and
Systems Sciences, CAS. We appreciate the hospitality and the
financial support from MCM. P. Zhang is partially supported by NSF
of China under Grant   11371347, the fellowship from Chinese Academy
of Sciences and innovation grant from National Center for
Mathematics and Interdisciplinary Sciences. Z. Zhang is partially
supported by NSF of China under Grant 11371037, 11421101, Program for New
Century Excellent Talents in University and Fok Ying Tung Education
Foundation.

\end{document}